\documentclass{article}

\usepackage[english]{babel}
\usepackage{amssymb}
\usepackage[all]{xy}
\usepackage{cite}

\newtheorem{definition}{Definition}[section]

\newtheorem{lemma}[definition]{Lemma}
\newtheorem{corollary}[definition]{Corollary}
\newtheorem{proposition}[definition]{Proposition}

\newtheorem{remark}[definition]{Remark}

\newenvironment{proof}{\vspace{3pt}\textsc{Proof:}\quad }
                       {\hfill \hbox{q.e.d.}\vspace{3pt}}

\def\B{{\mathcal B}}
\def\F{{\mathcal F}}

\def\P{{\mathcal P}}
\def\Q{{\mathcal Q}}

\def\sub{\subseteq }

\renewcommand{\land}{\mathrel\&}

\def\releps{\mathrel\epsilon}
\def\overlap{\, \between \,}
\newcommand{\sqoverlap}{>\mkern-13.5mu <}

\def\cov{\lhd}
\newcommand{\primconst}[1]
         {\mbox{\sf #1}}
\newcommand{\Pos}{{\primconst{Pos}}}

\newcommand{\fusim}{\cdot\mkern-4mu | \mkern-4mu\cdot}

\title{Constructive version of Boolean algebra}
\author{Francesco Ciraulo\footnote{Dipartimento
di Matematica, University of Padova, Via Trieste,
63 - I-35121 Padova, Italy, ciraulo@math.unipd.it}\ , Maria Emilia Maietti\footnote{Dipartimento
di Matematica, University of Padova, Via Trieste,
63 - I-35121 Padova, Italy, maietti@math.unipd.it}\ , Paola
Toto\footnote{Dipartimento di Matematica ``Ennio De Giorgi'',
University of Salento, Palazzo Fiorini, Via per Arnesano - I-73100
Lecce - Italy, paola.toto@unile.it}}
\date{}

\begin{document}

\maketitle
\begin{abstract}
The notion of overlap algebra introduced by G. Sambin provides a
constructive version of complete Boolean algebra. Here we first show
some  properties concerning overlap algebras: we prove that
the notion of overlap morphism corresponds classically to that of
map preserving arbitrary joins; we provide a description
of atomic set-based overlap algebras in the language of formal topology, thus giving a predicative characterization of  discrete locales;  we show that the power-collection of a set is the free overlap algebra join-generated from
the set.

Then, we generalize the concept of overlap algebra and overlap
morphism in various ways to provide constructive versions of the
category of Boolean algebras with
 maps preserving arbitrary existing joins.
\end{abstract}

\section{Introduction}
The classical Tarski's representation theorem (see \cite{tar35}),
asserting that atomic complete Boolean algebras coincide with
powersets, does not hold any longer if one works in a constructive
foundation. By a constructive  foundation we mean  one
governed by intuitionistic logic and enjoying
a semantics of extraction of programs from proofs (for a formal definition
see \cite{mtt}).
 The reason for the failure of Tarski's theorem is simple: when one drops the law of
excluded middle, powersets stop being Boolean algebras.
Hence, the following natural question arise:
\begin{enumerate}
\item
 what kind of algebraic structure characterizes
 powersets constructively?
\item
or better, what is the algebraic structure
corresponding classically to complete boolean algebras and including
constructive powersets  as examples?
\end{enumerate}
The first question was answered by Joyal and Tierney in \cite{joyal-tierney}
in a constructive impredicative foundation by giving a categorical  characterization of  {\it discrete locales},
while the second question was answered by Sambin in \cite{bp} within a constructive and predicative foundation
 by introducing the notion of {\it overlap algebra}. More precisely,
Sambin worked in the minimalist foundation introduced in  \cite{mtt,mtt2}
as a common core among the most relevant foundation for constructive mathematics. Since this foundation is predicative, in there one speaks of  {\it power-collections} and not of powersets, since   the power of a set, even that  of a singleton, is never a set predicatively.

An overlap algebra is a predicative locale
equipped with a notion of ``{\it overlap}'' between elements of the
algebra. The overlap relation is a positive way to express
when the meet of two elements is different from the bottom; among other things, it allows to define a suitable notion  of atom. Moreover, when overlap algebras are {\it set-based} (i.e. they have a set of
join-generators), they become in particular {\it formal topologies}~\cite{S87}.
The concept of formal topology was introduced by Martin-L{\"o}f and Sambin to describe locales predicatively, and it corresponds impredicatively to that of {\it overt} locale, and classically to that of locale.
But not all formal topologies are overlap algebras, because the overlap relation
is a proper strengthening of the  {\it positivity predicate} as shown by the fact  that overlap algebras coincide classically with complete Boolean algebras (see \cite{bp}). Actually,  in \cite{ci2} it is proven that  overlap
algebras coincide constructively with the collection of regular opens of formal topologies, thus giving a predicative version of the classical
representation theorem for complete Boolean algebras.
Finally, since Sambin in \cite{bp} proved that power-collections coincide with
 {\it atomic} set-based overlap algebras, we can conclude
that the notion of overlap algebra is the right constructive version of complete boolean algebra. Constructive examples of non-atomic overlap algebras are given in
\cite{cisa}.

In this paper, we show that Sambin's notion of overlap morphism in
\cite{bp} corresponds classically to that of map preserving
arbitrary joins, and hence the category of overlap
algebras in \cite{bp} is classically equivalent to that of complete
Boolean algebras with  maps preserving arbitrary joins. Furthermore,
by working in the minimalist foundation introduced in \cite{mtt2},
 we prove that the power-collection of all subsets of a set is
the free overlap algebra join-generated from the set.
Then,  we observe that we can present atomic set-based overlap algebras simply as suitable formal topologies, thus providing a predicative characterization of
 discrete locales  within the language of formal topology, instead of using the reacher language of overlap algebras.

Furthermore, we generalize the notion of overlap algebra and overlap
morphism to provide a constructive version of the category of (non
necessarily complete) Boolean algebras and maps preserving existing
joins. Basically we observe that join-completeness is not needed
when proving the equivalence between the category of overlap
algebras and that of complete Boolean algebras.

However, in order to represent
 boolean algebras constructively, we are faced with various
choices. Indeed we can  define different structures equipped with an
overlap relation with the same properties as the one given by Sambin
but related only to existing joins: we define a Boolean algebra with
overlap, called {\it o-Boolean algebra}, a Heyting algebra with
overlap,  called {\it o-Heyting algebra}, and a lattice with an
opposite (pseudocomplement)  and  overlap, called {\it oo-lattice}.
We show that such structures with overlap, for short {\it
o-structures}, classically (and impredicatively) are nothing but
Boolean algebras. Constructively, we show that they are all different and we study their mutual relationships.

As a future work we intend to investigate whether we can use our o-structures to obtain constructive versions of classical representation
theorems for  Boolean algebras.

\section{Some remarks on foundations}
When developing our theorems we assume to work in the extensional
set theory
 of the two-level minimalist foundation in \cite{mtt2}. This was designed
according to the principles given in \cite{mtt}.
The main characteristic is that our foundational set theory is {\it
constructive} and {\it predicative}. The fact that
our foundation is constructive means that
it is governed by intuitionistic logic, which does
not validate excluded middle, and it enjoys a realizability model
where to extract programs from proofs. The predicativity of our foundation
implies  that the power-collection of subsets of a set $X$,  written
$\P(X)$, is not a set, but a proper collection. Hence in
our set theory we have the notion of set and that of collection. To
keep predicativity,  a subset of a set $X$ can {\it only} be defined
by comprehension on a formula $\varphi(x)$, for $x\releps X$, with
quantifiers restricted to sets; such a subset is written $\{x\releps
X \ \mid\   \varphi(x)\}$.

It is worth mentioning that a  complete join-semilattice (also called suplattice) whose carrier is a
set is necessarily trivial  in a predicative constructive foundation
\cite{Cu10}. Therefore in such a setting we are
lead to define a complete join-semilattice as  a collection closed
under joins of \emph{set-indexed} families (we cannot assume
\emph{arbitrary} joins to exist, otherwise we fall again into a
trivial lattice).

As done  in \cite{S87}
 and \cite{bat}
we can make the definition of complete join-semilattice easier to
handle by restricting ourselves to the notion of \emph{set-based}
complete join-semilattice, given that all the relevant predicative
 examples known so far fall under this class.
A set-based complete join-semilattice is
 a semilattice that is
join-generated from a set(-indexed family) of elements, called
\emph{join-generators}. This means that each element
is the join of all the join-generators below it. Such a join exists if
 the collection of all join-generators
below an element form a subset (equivalently, a set-indexed
family). In order to achieve this, we need the order of the
semilattice to be defined by a formula containing only
quantifications over sets. For instance, the order in $\P(X)$ making it a set-based complete join-semilattice is written as
follows: $A\sub B$ iff $(\forall x\releps X)(x\releps A\Rightarrow
x\releps B)$.

Note that every set-based suplattice has binary meets, also predicatively. In fact, one can construct $x\wedge y$ has the join of all generators $a$ such that both $a\leq x$ and $a\leq y$ hold. This is well-defined in view of the discussion above.

In this paper we will deal with overlap algebras, which are in
particular complete join-semilattices, and we assumed them to be all
set-based.

Before starting, let us agree on some notation: $X$, $Y$, $S$ and
$T$ will always denote sets with elements $x$, $y$, $z$,\ldots $a$,
$b$, $c$, \ldots and subsets $A$, $B$, $C$, $D$, $E$,\ldots $U$,
$V$, $W$, $Z$. On the contrary, $\P$ and $Q$ will always stand for
collections whose elements will be written as $p$, $q$, $r$, \ldots.
Accordingly, we will use two different symbols to distinguish between the
two kinds of membership: $\releps$ for sets and subsets, $:$ for
collections.

\section{Overlap algebras}
The notion of overlap algebra   was introduced by Sambin in
\cite{bp}. It provides a constructive and predicative version of complete boolean
algebra and it includes constructive power-collections of sets as examples.
 Of course, the usual notion of
complete Boolean algebra is not apt to this purpose because, in a
constructive foundation, power-collections are not Boolean algebras
but only complete Heyting algebras. An overlap algebras is an enrichment of the notion of predicative locale with an {\it overlap relation}
used to define positively when the meet of two elements is different from the bottom.
In the case of power-collections, the notion of overlap between two
subsets $A,B\sub X$, denoted by $A\overlap B$, expresses {\it
inhabitedness} of their intersection and is therefore defined as
follows:
\begin{equation}\label{eq. def. overlap}
A\overlap B\ \stackrel{def}{\Longleftrightarrow}\ (\exists x\releps
X)(x\releps A\cap B)\ .
\end{equation}
Moreover, classically, any complete Boolean algebra $\cal B$ is  equipped
with an overlap relation defined as $x\wedge y\neq 0$ for
$x,y\releps {\cal B}$ (see \cite{bp}).
 Thus the notion of overlap is
a constructive positive way to express {\it inhabitedness} of the
meet of two elements. As we will see, it allows to define the notion of atom.
Atomic set-based overlap algebras coincide constructively with power-collections
of sets. Therefore the notion of atomic set-based overlap algebra
 provides a predicative version of the categorical characterization of discrete locales in \cite{joyal-tierney} within the language of overlap algebras.
Here, after reviewing some basic facts on overlap algebras,  we will describe
atomic set-based overlap algebras in terms of formal topologies, thus providing
a predicative version of
discrete locales within the language of formal topology. We then end by showing  that
the power
collection of a set is the free overlap algebra join-generated from
the set.

\subsection{Definition and basic properties}

\begin{definition}\label{def. o-algebras}
An \emph{overlap algebra} (\emph{o-algebra} for short) is a triple
$(\P,\leq,\sqoverlap)$ where $(\P,\leq)$ is a suplattice and
$\sqoverlap$ is a binary relation on $\mathcal{P}$ satisfying the
following properties:
  \begin{itemize}
    \item $p\sqoverlap q\ \Rightarrow\ q\sqoverlap p$\hfill (symmetry)
    \item $p\sqoverlap q\ \Rightarrow\ p\sqoverlap (p\wedge q)$\hfill (meet closure)
    \item $p\sqoverlap\bigvee_{i\releps I}q_i\ \Longleftrightarrow\ (\exists i\releps I)(p\sqoverlap q_i)$\hfill (splitting of join)
    \item $(\forall r:\mathcal{P})(r\sqoverlap p\ \Rightarrow\ r\sqoverlap q)\ \Longrightarrow\ p\leq q$\hfill(density)
  \end{itemize} (for any $p$ and $q$ in $\P$).
\end{definition}

We say that an o-algebra $(\P,\leq,\sqoverlap)$ is \emph{set-based}
if the join-semilattice $(\P,\leq)$ admits a \emph{base}, that is, a
set-indexed family of generators (with respect to the operation of
taking set-indexed joins), called \emph{join-generators}. We agree
to make no notational distinction between the base and its index
set; thus $S$ is a base for $\P$ if $p=\bigvee\{a\releps S$ $|$
$a\leq p\}$ for any $p:\P$. For the reasons mentioned in section 2,
we shall assume each o-algebra to be set-based.

It is easily seen that all quantifications over the elements of an
o-algebra $\P$ can be reduced to the base. For instance, the
``density'' axiom in definition \ref{def. o-algebras} is equivalent
to
\begin{equation}\label{eq. set-based density}(\forall
a\releps S)(a\sqoverlap p\ \Rightarrow\ a\sqoverlap q)\
\Longrightarrow\ p\leq q\end{equation} ($S$ being a base). Just to
get acquainted with the axioms, let us prove that ``density''
implies (\ref{eq. set-based density}) (the other direction being
trivial). It is enough to prove that $(\forall a\releps
S)(a\sqoverlap p$ $\Rightarrow$ $a\sqoverlap q)$ $\Longrightarrow$
$(\forall r:\mathcal{P})(r\sqoverlap p$ $\Rightarrow$ $r\sqoverlap
q)$. Take $r:\P$ such that $r\sqoverlap p$. Since
$r=\bigvee\{a\releps S$ $|$ $a\leq r\}$ and $\sqoverlap$ splits
$\bigvee$, there exists $a\releps S$ such that $a\leq r$ and
$a\sqoverlap p$. By hypothesis, we get $a\sqoverlap q$ and hence
$r\sqoverlap q$ by the splitness property again.

For every set $X$, the structure $(\P(X),\sub,\overlap)$, where $\overlap$
is defined as in
equation (\ref{eq. def. overlap}), is an o-algebra whose singletons form a base. In
addition, we shall see in the following sections that $\P(X)$ is
also atomic and free over $X$. Here below, we list some useful
properties of o-algebras. Detailed proofs can be found in \cite{bp}, \cite{ci2} and \cite{cisa}.

\begin{proposition}
\label{properties o-algebra} Let $\P$ be an o-algebra with base $S$; then the following hold:
\begin{enumerate}
    \item $p\sqoverlap r\ \land\ r\leq q\ \Longrightarrow\ p\sqoverlap q$
    \item $p=q\ \Longleftrightarrow\ (\forall a\releps S)(a\sqoverlap p\
    \Leftrightarrow\ a\sqoverlap q)$
    \item $(p\wedge r)\sqoverlap q\ \Longleftrightarrow\ p\sqoverlap (r\wedge q)$
    \item $p\sqoverlap q\ \Leftrightarrow\ (p\wedge q)\sqoverlap(p\wedge q)\
    \Leftrightarrow\ (\exists a\releps S)(a\leq p\wedge
    q\ \land\ a\sqoverlap a)$
    \item $\neg(0\sqoverlap 0)$
    \item $\neg(p\sqoverlap q)$ $\Longleftrightarrow$ $p\wedge q=0$
\end{enumerate}
for every $p$, $q$, $r$ in $\P$.
\end{proposition}
\begin{proof} (1) From
$p\sqoverlap r$, it follows that $p\sqoverlap (r\vee q)$ by
splitness; but $r\vee q = q$ because $r\leq q$. (2) By density. (3) By meet closure and item 1. (4) By meet closure, symmetry and
splitness of join. (5) By splitness of join, because $0$ is the join of the empty family.
(6) If $p\wedge q=0$, then $p\sqoverlap q$ would contradict item 5
(by item 4). To prove the other direction we use the density axiom:
for $a\releps S$, if $a\sqoverlap (p\wedge q)$, then $p\sqoverlap q$
(by symmetry and items 4 and 1) which contradicts the assumption
$\neg(p\sqoverlap q)$, so $a\sqoverlap 0$ (\emph{ex falso quodlibet}).
\end{proof}

As a corollary, we get that all the axioms in Definition \ref{def.
o-algebras} can be reversed. In particular, the order relation can
be considered as a defined notion with the overlap relation as
primitive (thanks to the axiom ``density'' and its converse). Not surprisingly, most
of the times an inequality has to be proven, we shall apply
``density'' (as we have already done in the proof of item 6);
similarly, for equalities we shall often use item 2 above.

Note that in a set-based o-algebra the overlap is {\it uniquely} determined.
First, observe that by density  and item 4 of prop.~\ref{properties o-algebra}
 $p= \bigvee\{a\releps
S$ $|$ $a\leq p\ \&\  a\sqoverlap a \}$ holds for all $p$ in  a set-based o-algebra $\P$ with base $S$. 
Suppose that $\sqoverlap_1$ and $\sqoverlap_2$ are overlap relations in $\P$. We now show that $\sqoverlap_1$ and 
$\sqoverlap_2$ are equivalent. To this purpose it is enough to prove that $\sqoverlap_1$ implies $\sqoverlap_2$.
Suppose $p \sqoverlap_1 q$, then by meet closure $p \sqoverlap_1 p\wedge q$; now
from $p\wedge q= \bigvee\{a\releps
S$ $|$ $a\leq p\wedge q\ \&\  a\sqoverlap_2 a \}$
by splitness of join we deduce that
there exists $a\releps S$ such that  $p \sqoverlap_1 a$,
 $a\leq  p\wedge q$,  $a\sqoverlap_2 a$, and by item 1 of  prop.~\ref{properties o-algebra} and symmetry we conclude
 $p \sqoverlap_2 q$.

The intuition underlying the relation $\sqoverlap$ suggests that
there should be deep links between it and the positivity predicate
in Formal Topology. The notion of formal topology was introduced by
Martin-L{\"o}f and Sambin to describe locales predicatively; it
corresponds impredicatively to that of {\it open} locale, and
classically to that of locale (see \cite{S87,Joh,Sam2002}). Here we
give a definition of a formal topology which is more suitable for
our purposes, though  equivalent to that in \cite{cssv}.

\begin{definition}
 A \emph{formal topology} ${\P}$ is, first of all, a \emph{formal cover}, namely
 a set-based suplattice in which binary meets distribute over set-indexed
joins, that is:\begin{equation}\label{eq. infinite distributivity}
p\wedge\bigvee_{i\releps I}q_i\quad=\quad\bigvee_{i\releps I}(p\wedge q_i)\
.\end{equation} If $S$ is a base, $a\releps S$ and $U\sub S$, one usually writes
$$a\cov U \qquad \mbox{ for }\qquad a\leq\bigvee\{a\releps S\ |\ a\releps U\}\ .$$

Moreover, $\P$  is equipped with
 a \emph{positivity predicate}, namely a unary predicate $\Pos(p)$ for $p$ in $\P$
such that: for $p, q$ and $p_i$ (for $i\releps I$) in $\cal P$
\\

\begin{tabular}{ll}
$\Pos(p)$ $\land$ $p\leq q$ $\Longrightarrow$ $\Pos(q)$&
$\Pos\big(\bigvee_{i\releps I}q_i\big)$ $\Longrightarrow$ $(\exists\,
i\releps I)$ $\Pos(q_i)$\\[5pt]
 $p\leq\bigvee\{a\releps S\ |\ a\leq p\ \land \ \Pos(a)\, \}$
& (the so-called ``positivity axiom'').
\end{tabular}
\end{definition}
In the following we simply say that an element $p$ of a formal topology $\P$ is \emph{positive} if it satisfies $\Pos(p)$.
\begin{proposition}
\label{prop o-lagebras vs overt locales} Every set-based o-algebra is a formal topology with the positivity predicate defined as
$p\sqoverlap p$.
\end{proposition}
\begin{proof} Firstly, we claim that every o-algebra $\P$ is, in fact, a formal cover;
we need only to prove (\ref{eq. infinite distributivity}). For all
$r:\P$ the following hold: $r\sqoverlap (p\wedge\bigvee_{i\releps
I}q_i)$ iff $(r\wedge p)\sqoverlap\bigvee_{i\releps I}q_i$ iff $(r\wedge
p)\sqoverlap q_i$ for some $i\releps I$ iff $r\sqoverlap (p\wedge
q_i)$ for some $i\releps I$ iff $r\sqoverlap\bigvee_{i\releps I}(p\wedge
q_i)$. This is sufficient by (2) of proposition~\ref{properties o-algebra}.

Now, let us put $\Pos(p)\Leftrightarrow(p\sqoverlap p)$. We claim that
$\Pos$ is a positivity predicate. All the requested properties are
quite easy to prove (hint: use density to prove the positivity axiom).
\end{proof}

The notion of o-algebra is stronger than that of formal topology. Indeed, first of all,  for a formal topology
$\P$ the binary
predicate $\Pos(x\wedge y)$ satisfies all the axiom of an overlap
relation but density. In other words, a formal topology is an
o-algebras if and only if its positivity predicate $\Pos$ satisfies
the following: $$(\forall a\releps S)\big(\Pos(a\wedge
p)\Rightarrow\Pos(a\wedge q)\big)\Longrightarrow p\leq q$$ for every
$p$ and $q$.
Moreover, in the next
we are going to prove that o-algebras coincide classically with complete Boolean
algebras and hence, of course, they are ``fewer'' than formal topologies, or locales.

\begin{remark}
Since we are assuming in this paper that every o-algebra is set-based,
it follows that every
o-algebra is a complete Heyting algebra, because it has an  implication
 defined by: $p\rightarrow q$ = $\bigvee\{a\releps S$ $|$ $a\wedge p\leq
q\}$, where $S$ is a base. The validity of this fact
is one of the advantages of working with set-based structures.
\end{remark}

\subsection{O-algebras classically}

Building on item 6 of Proposition \ref{properties o-algebra}, we now
state two lemmas, essentially due to Sambin (see \cite{bp}), which
further clarify the relationship between the overlap relation
$p\sqoverlap q$ and its negative counterpart $p\wedge q\neq 0$.

\begin{lemma}\label{lemma 1}
Classically, in any o-algebra, $p\sqoverlap q$ is tantamount to
$p\wedge q\neq 0$.
\end{lemma}
\begin{proof}
From Proposition~\ref{properties o-algebra}, item 6.
\end{proof}

\begin{lemma}\label{lemma 2} Let $(\P,\wedge,0,-)$ be a
$\wedge$-semilattice with bottom and with a
pseudocomplement, i.e. a unary operation
$-$ such that $p\leq-q$ if and only if $p\wedge q=0$. The following
are equivalent constructively: \begin{enumerate}
\item $(\forall p,q:\P)\big((\forall r:\P)(r\wedge p\neq 0$
$\Rightarrow$ $r\wedge q\neq 0)$ $\Longrightarrow$ $p\leq q\big)$
(\emph{negative density}); \item $(\forall p:\P)(p=--p)$ $\land$
$(\forall p,q:\P)\big(\neg(p\neq q)$ $\Rightarrow$ $p=q\big)$.
\end{enumerate}
\end{lemma}
\begin{proof} Assume $1.$ and let $r$ be
such that $r\wedge --p\neq 0$, that is, $r\nleq---p$. Since
$---p=-p$, this is tantamount to say that $r\nleq-p$, that is,
$r\wedge p\neq 0$. Since $r$ is arbitrary, we get $--p\leq p$ by
negative density. Hence $--p=p$ for any $p:\P$. Assume now
$\neg(p\neq q)$; we claim that $p=q$. It is enough to prove that
$p\leq q$; so, by negative density, we must check that $r\wedge
p\neq 0$ implies $r\wedge q\neq 0$: this is easy because if it were
$r\wedge q=0$, then it would be $p\neq q$ (since $r\wedge p\neq 0$),
contrary to the assumption $\neg(p\neq q)$.

Vice versa, assume $2.$ and note that the implication $r\wedge p\neq 0$
$\Rightarrow$ $r\wedge q\neq 0$ can be rewritten as $\neg(r\wedge q\neq 0)$
$\Rightarrow$ $\neg(r\wedge p\neq 0)$ which, by hypothesis, is
equivalent to $r\wedge q=0$ $\Rightarrow$ $r\wedge p=0$. Thus the
antecedent of negative density becomes $(\forall r:\P)(q\leq-r$
$\Rightarrow$ $p\leq-r)$; the latter gives in particular (even is
equivalent to) $p\leq q$ (choose $r=-q$ and use $--q=q$).
\end{proof}

We think it is illuminating to compare o-algebras with complete
Boolean algebras (this result was first suggested by Steve Vickers).

\begin{proposition}
\label{prop. o-algebras cBa's} Classically,
every complete Boolean algebra (with $0\neq 1$) is an o-algebra
(with $1\sqoverlap 1$), where $p\sqoverlap q$ $\Leftrightarrow$
$p\wedge q\neq 0$.

Classically and impredicatively, every
o-algebra (with $1\sqoverlap 1$) is a complete Boolean algebra (with
$0\neq 1$).
\end{proposition}

\begin{proof} (See \cite{bp} and \cite{ci2}). Start with a complete
Boolean algebra and (consider Lemma \ref{lemma 1}) define
$p\sqoverlap q$ as $p\wedge q\neq 0$. This relation trivially
satisfies symmetry and meet closure. Density follows from the
previous Lemma. Finally, splitting of join can be easily reduced to
$\neg(\forall i\releps I)(p\wedge q_i= 0)$ $\Leftrightarrow$ $(\exists
i\releps I)(p\wedge q_i\neq 0)$ which is classically valid.

Conversely, every set-based o-algebra is a complete Heyting algebra,
as we know. Moreover, by Lemma \ref{lemma 1} and Lemma \ref{lemma
2}, we get that $-- q= q$, for any $q$. Finally, the powerset axiom
allows one to consider the carrier of an o-algebra as a set, as required
by the usual definition of complete Boolean algebra.
\end{proof}

\subsection{Morphisms between o-algebras}

\begin{definition} Let $f:\P\rightarrow\Q$ and
$g:\Q\rightarrow\P$ be two maps between o-algebras. We say that $f$
and $g$ are \emph{symmetric}\footnote{This notion is classically
equivalent to that of ``conjugate'' functions studied in
\cite{tarski}.} and we write $f\fusim g$ if
$$f(p)\sqoverlap q\ \Longleftrightarrow\ p\sqoverlap g(q)$$ for all
$p:\P$ and $q:\Q$.
\end{definition}

In \cite{bp} Sambin proposed and widely justified the following
definition of morphism between o-algebras.

\begin{definition}\label{def. o-morphism}
An \emph{overlap morphism} (\emph{o-morphism}) from an o-algebra
$\P$ to an o-algebra $\Q$ is a map $f:\P\rightarrow\Q$ such that
there exist $f^-,f^*:\Q\rightarrow\P$ and $f^{-*}:\P\rightarrow\Q$
satisfying the following conditions:
\begin{enumerate}
    \item $f(p)\leq q\ \Longleftrightarrow\ p\leq f^*(q)$\hfill ($f\dashv f^*$)
    \item $f^-(q)\leq p\ \Longleftrightarrow\ q\leq f^{-*}(p)$\hfill ($f^-\dashv f^{-*}$)
    \item $f(p)\sqoverlap q\ \Longleftrightarrow\ p\sqoverlap f^-(q)$\hfill ($f\fusim f^-$)
\end{enumerate} (for all $p:\P$ and $q:\Q$).
\end{definition}
Easily, the identity map $id_\P$ on $\P$ is an o-morphism (with
$id_\P^-$ $=$ $id_\P^*$ $=$ $id_\P^{-*}$ $=$ $id_\P$); moreover, the
composition $f\circ g$ of two o-morphisms is an o-morphism too (with
$(f\circ g)^-=g^-\circ f^-$, $(f\circ g)^*=g^*\circ f^*$ and
$(f\circ g)^{-*}=f^{-*}\circ g^{-*}$).

\begin{definition}
O-algebras and o-morphisms form a category, called $\mathbf{OA}$.
\end{definition}

Here we present an example of o-morphism which is actually the
motivating one. For $X$ and $Y$ sets, it is possible (see \cite{bp})
to characterize o-morphisms between the overlap algebras $\P(X)$ and
$\P(Y)$ in terms of binary relations between $X$ and $Y$. For any
relation $R$ between $X$ and $Y$, consider its \emph{existential
image} defined by
$$R(A)\quad\stackrel{def}{=}\quad\{y\releps Y\ |\ (\exists x\releps
X)(x\,R\,y\ \land\ x\releps A)\}$$ (for $A\sub X$). It is easy to
check that the operator $R$ is an o-morphism from $\P(X)$ to $\P(Y)$
with $R^-$, $R^*$ and $R^{-*}$ defined, respectively, by:
\begin{eqnarray}\label{eq. induced operators}
R^-(B) & \stackrel{def}{=} & \{x\releps X\ |\ (\exists y\releps
Y)(x\,R\,y\ \land\
y\releps B)\}\nonumber\\
R^*(B) & \stackrel{def}{=} & \{x\releps X\ |\ (\forall y\releps
Y)(x\,R\,y\ \Rightarrow\
y\releps B)\}\\
R^{-*}(A) & \stackrel{def}{=} & \{y\releps Y\ |\ (\forall x\releps
X)(x\,R\,y\ \Rightarrow\ x\releps A)\}\nonumber
\end{eqnarray} (for any $A\sub X$ and $B\sub Y$).  Vice
versa, every o-morphism $f:\P(X)\rightarrow\P(Y)$ is of this kind: define $x\,R\,y$ as $y\releps f(\{x\})$. This
correspondence is biunivocal and defines a full embedding of the
category of sets and relations  into $\mathbf{OA}$.

 The conditions of an o-morphism
 simplify in a relevant way in the case its domain and codomain are set-based o-algebras. First of all,
one should recall from category theory that, considering $f$ and $f^-$
as functors
between the poset categories $\P(X)$ and $\P(Y)$ (since they are monotone functions),  the functions $f^*$ and $f^{-*}$ are the right adjoints of them respectively, and hence, if exists, they are  uniquely determined by $f$  and
$f^-$. Moreover,  $f^*$
(respectively $f^{-*}$) exists if and only if $f$ (respectively
$f^-$) preserves all joins. This is true in an impredicative
setting, but also predicatively at least for set-based structures.
In the latter case $f^*(q)$ can be defined as $\bigvee\{a\releps S$
$|$ $f(a)\leq q\}$ (and similarly for $f^{-*}$). Before going on,
let us prove a few properties about symmetric functions.

\begin{proposition}\label{prop. properties symmetric}
Let $f$ be a map on the o-algebra $\P$ to the o-algebra $\Q$ such that there exists $g:\Q\rightarrow\P$ with $f\fusim g$; then:
\begin{enumerate}
  \item $g$ is unique; that is, if $h:\Q\rightarrow\P$ satisfies $f\fusim h$, then $h=g$;
  \item $g$ is determined by $f$, in the sense that for any $q:\Q$
  \begin{equation}\label{eq. def. f^-}g(q)=\bigvee\{a\releps S\ |\ (\forall x\releps S)(x\sqoverlap a\
  \Rightarrow\ f(x)\sqoverlap q)\}\end{equation}
  where $S$ is a base for $\P$.
\end{enumerate}
\end{proposition}
\begin{proof}
(1) For every $x$ in (a base for) $\P$ and every $y$ in $\Q$, we have: $x\sqoverlap h(y)$
iff $f(x)\sqoverlap y$ iff $x\sqoverlap g(y)$. Hence $h=g$ by (2) of prop.~\ref{properties o-algebra}. (2): $g(q)$ $=$ $\bigvee\{a\releps S$ $|$
  $a\leq g(q)\}$ $=$ $\bigvee\{a\releps
S$ $|$
  $(\forall x\releps S)(x\sqoverlap a$ $\Rightarrow$ $x\sqoverlap g(q)\}
  =\bigvee\{a\releps
S$ $|$
  $(\forall x\releps S)(x\sqoverlap a$ $\Rightarrow$ $f(x)\sqoverlap
  q)\}$.
\end{proof}

\begin{definition} We say that a map $f:\P\rightarrow\Q$ between o-algebras is
\emph{symmetrizable} if there exists a (necessarily unique) map
$f^-:\Q\rightarrow\P$ such that $f\fusim f^-$. In that case, we say
that $f^-$ is \emph{the symmetric} of $f$.
\end{definition}

\begin{remark} Since $\sqoverlap$ is a symmetric binary relation, if $f$ is
symmetrizable also $f^-$ is and $(f^-)^-=f$. Note also that, if
$f:\P\rightarrow\Q$ is an o-morphism, then also
$f^-:\Q\rightarrow\P$ is an o-morphism.\footnote{Observe that $f^-$ is in fact the inverse of $f$ when the latter is an
isomorphism.}
\end{remark}

\begin{lemma}\label{prop. symmetrizable preseves joins}
Let $f$ be a symmetrizable map on the o-algebra $\P$ to the
o-algebra $\Q$; then $f$ (and $f^-$) preserves all (set-indexed)
joins.
\end{lemma}
\begin{proof} For any $y:\P$, we have:
$y\sqoverlap f(\bigvee_{i\releps I}p_i)$ iff
$f^-(y)\sqoverlap\bigvee_{i\releps I}p_i$ iff $(\exists i\releps
I)(f^-(y)\sqoverlap p_i)$ iff $(\exists i\releps I)(y\sqoverlap
f(p_i))$ iff $y\sqoverlap\bigvee_{i\releps I}f(p_i)$. Hence (by
density) we can conclude that $f(\bigvee_{i\releps
I}p_i)=\bigvee_{i\releps I}f(p_i)$.
\end{proof}

\begin{proposition}\label{prop. char. o-morphism}
Let $f:\P\rightarrow\Q$ be a map between two (set-based) o-algebras;
then the following are equivalent:
\begin{enumerate}
\item $f$ is an o-morphism;
\item $f$ is symmetrizable;
\item $f$ satisfies the following property:
\begin{equation}\label{RED}
f(p)\sqoverlap q\ \Longleftrightarrow\ (\exists a\releps
S)\big(p\sqoverlap a \ \land \ (\forall x\releps S)(x\sqoverlap a\
\Rightarrow\ f(x)\sqoverlap q)\big)
\end{equation} for all $p:\P$ and $q:\Q$ (where $S$ is a base of $\P$).
\end{enumerate}
\end{proposition}
\begin{proof}
(3 $\Rightarrow$ 2) We show that the function $g(q)$ $=$
$\bigvee\{a\releps S$ $|$ $(\forall x\releps S)(x\sqoverlap a$
$\Rightarrow$ $f(x)\sqoverlap q)\}$ of Proposition \ref{prop.
properties symmetric} is in fact the symmetric of $f$. As
$\sqoverlap$ splits joins, we have $p\sqoverlap g(q)$ if and only if
$p\sqoverlap a$ for some $a$ satisfying $(\forall x\releps
S)(x\sqoverlap a$ $\Rightarrow$
$f(x)\sqoverlap q)$ and this holds if and only if, by 3, $f(p)\sqoverlap q$.\\
(2 $\Rightarrow$ 1) By Proposition \ref{prop. symmetrizable preseves
joins}, both $f$ and $f^-$ preserve joins; hence their right
adjoints
$f^*$ and $f^{-*}$ exist.\\
(1 $\Rightarrow$ 3) Let $f^-$ be the symmetric of $f$. Then
$(\exists a\releps S)(p\sqoverlap a$ $\land$ $(\forall x\releps
S)(x\sqoverlap a$ $\Rightarrow$ $f(x)\sqoverlap q))$ iff $(\exists
a\releps S)(p\sqoverlap a$ $\land$ $(\forall x\releps S)(x\sqoverlap
a$ $\Rightarrow$ $x\sqoverlap f^-(q)))$ iff $(\exists a\releps
S)(p\sqoverlap a$ $\land$ $a\leq f^-(q))$ iff \footnote{Because
$p\sqoverlap q$ $\Longleftrightarrow$ $p\sqoverlap\bigvee\{a\releps
S$ $|$ $a\leq q\}$ $\Longleftrightarrow$ $(\exists a\releps
S)(p\sqoverlap a\ \land\ a\leq q)$.} $p\sqoverlap f^-(q)$ iff
$f(p)\sqoverlap q$.
\end{proof}

Here we want to spend some words about item 3. Firstly, it is surely
of some interest because it characterizes the notion of o-morphism
by an intrinsic property of the map $f$ itself. Moreover, it seems
the right notion of morphism in the non-complete case, as we shall
see in the last sections. Furthermore, we think it is worth
mentioning that (\ref{RED}) is a form of continuity. This fact is
better seen in the context of formal topology (see \cite{ci2}).
However, we can here give a suggestion: following equations
(\ref{eq. induced operators}) we write $(\sqoverlap q)$ for $\{\, x\releps S\, \mid\,  x\sqoverlap q\}$; then condition (\ref{RED}) can be rewritten
as $f^{-1}(\sqoverlap q)$ $=$ $\bigcup\{(\sqoverlap a)$ $|$
$(\sqoverlap a)\sub f^{-1}(\sqoverlap q)\}$. Thus, if the families
$\{\sqoverlap p\}_{p:\P}$ and $\{\sqoverlap q\}_{q:\Q}$ are taken as
sub-bases for two topologies on $\P$ and $\Q$, respectively, then
(\ref{RED}) is a notion of continuity for $f$ (in fact, this is
stronger than usual continuity because the $(\sqoverlap p)$'s do not
form a base, in general). By the way, note that reading $\sqoverlap$
as a unary operator allows to rewrite $f\fusim f^-$ as
$f^{-1}\circ\sqoverlap$ $=$ $\sqoverlap\circ f^-$.

\begin{proposition}\label{prop. CLASS: o-morphism = join-preserving}
Classically and impredicatively, o-morphisms are exactly the maps
preserving all joins.
\end{proposition}
\begin{proof}
In a classical setting, an o-algebra is exactly a cBa (Proposition
\ref{prop. o-algebras cBa's}). As we already know from Proposition
\ref{prop. symmetrizable preseves joins}, every o-morphism is
join-preserving. Viceversa, if $f:\P\rightarrow\Q$ preserves all
joins, then (by the powerset axiom) it admits a right adjoint $f^*$.
We claim that $f^-$ exists and it is $f^-(q)=-f^*(-q)$. Indeed, for
$p$ in $\P$ we have: $p\sqoverlap -f^*(-q)$ iff $p\wedge-f^*(-q)\neq
0$ iff $p\nleq f^*(-q)$ iff $f(p)\nleq -q$ iff $f(p)\wedge q\neq 0$
iff $f(p)\sqoverlap q$.
\end{proof}

\begin{definition}
Let $\mathbf{cBa}$ be the category of complete Boolean algebras with maps
preserving arbitrary joins and finite meets (and hence complements).

We write $\mathbf{cBa}_{\bigvee}$ for the category (of which $\mathbf{cBa}$ is a subcategory) of complete Boolean algebras and
\emph{join-preserving maps} (maps which preserve arbitrary joins).
\end{definition}

\begin{corollary} Classically, the categories $\mathbf{OA}$ and
$\mathbf{cBa}_{\bigvee}$ are equivalent.
\end{corollary}
\begin{proof}
By Propositions \ref{prop. o-algebras cBa's} and \ref{prop. CLASS:
o-morphism = join-preserving}.
\end{proof}

Symmetrically, it is not difficult to select a subcategory of
$\mathbf{OA}$ which is isomorphic to the whole of $\mathbf{cBa}$.

\begin{definition}
Let $\mathbf{OA}^{\wedge}$ be the subcategory of $\mathbf{OA}$ with
the same objects as $\mathbf{OA}$ and whose morphisms are the
o-morphisms preserving finite meets.
\end{definition}

\begin{corollary}
The category $\mathbf{cBa}$ of complete Boolean algebras is
classically equivalent to the category $\mathbf{OA}^{\wedge}$.
\end{corollary}

\subsection{Atomic o-algebras as discrete formal topologies}

In a poset with zero, every minimal non-zero element is usually
called an ``atom''. Here, we see how to define the notion of atom
in an overlap algebra. This will allow us to characterize power-collections as atomic (set-based
o-algebras). Then, after noting that the notion of atom can be also given
in the language of formal topology, we define the notion of  {\it discrete formal topology}
and compare it with the categorical characterization of discrete locales by Joyal and Tierney in \cite{joyal-tierney}.

\begin{definition}[atom in an o-algebra]
\label{def. atom} Let $\P$ be an overlap algebra. We say that
$m:\mathcal{P}$ is an \emph{atom} if $m\sqoverlap m$ and for every
$p:\mathcal{P}$, if $m\sqoverlap p$ then $m\leq p$.
\end{definition}

There are several useful characterization of this notion; among
them, we list the following.

\begin{lemma}\label{lemma char. atoms}
In any o-algebra $\P$, the following are equivalent:
\begin{enumerate}
\item $m$ is an atom;
\item $m\sqoverlap m$ and, for every $p:\P$, if $p\sqoverlap p$ and
$p\leq m$, then $p=m$;
\item for every
$p:\P$, $m\sqoverlap p$ if and only if $m\leq p$;
\item $m\sqoverlap m$ and, for every $p,q:\P$, if $m\sqoverlap
p$ and $m\sqoverlap q$, then $m\sqoverlap p\wedge q$.
\end{enumerate}
\end{lemma}
\begin{proof} See \cite{bp} and \cite{ci2}. \end{proof}

\begin{definition}
We say that an overlap algebra is \emph{atomic} if its atoms form a base, i.e.
the atoms form a set and each element of the algebra
is join-generated from a subset of atoms.
\end{definition}
 Clearly $\P(X)$ is
atomic; and this is, essentially, the only example (see \cite{bp}).

\begin{proposition}\label{prop. atomic o-alg. = powerset}
\label{prop. sadocco-sambin} An o-algebra $\mathcal{Q}$ is atomic
if and only if it is isomorphic to $\mathcal{P}(S)$, for some set
$S$.
\end{proposition}
\begin{proof} One shows that an atomic o-algebra $\Q$
is isomorphic to ${\cal P}(A)$ where $A$ is the set of atoms of $\Q$.
  \end{proof}

Note that the definition of atom given above makes sense also for the more general notion of a formal topology. In particular, item 2 above suggests the following:

\begin{definition}[atom in a formal topology]
\label{atomft}
Let $\P$ be a formal topology with base $S$ and positivity predicate $\Pos$. We say that an element $a\releps S$ is an \emph{atom}, written $a \releps At(\P)$ if
\begin{enumerate}
\item $\Pos(a)$ holds and
\item for every other $b\releps S$, if $\Pos(b)$ and $b\leq a$, then $b=a$.
\end{enumerate}
\end{definition}

This definition of atom is predicative (no quantification over collections is required), but not restrictive. In fact, it is easy to see that:
\begin{itemize}
\item if $p:\P$ satisfies $1.$ and $2.$ above, then $p\releps S$ (in the usual notation of formal topology this is trivial: any subset $U$ of $S$ that is an atom is a singleton, because of $2.$ and because, being $U$ positive, it is certainly inhabited);
\item if $a$ is an atom, then $2.$ is satisfied for all $b:\P$ (even not belonging to $S$).
\end{itemize}

This definition captures the usual intuition on atoms. For instance, it is easy to check that: \emph{if $a\releps At(\P)$ and $a\leq \bigvee_{u\releps U}u$, then $a\leq u$ for some $u\releps U$}. In fact, from $a\leq\bigvee_{u\releps U}u$ on gets $a\leq a\wedge\bigvee_{u\releps U}u =\bigvee_{u\releps U}(a\wedge u)$. From $\Pos(a)$ one thus obtains $\Pos(a\wedge u)$ for some $u\releps U$. But $a\wedge u\leq a$ and so $a\wedge u=a$, that is, $a\leq u$.

Since the collection of all atoms as defined in definition~\ref{atomft}
form a subset $At(\P)$ of $S$, and hence a set,   we can give a predicative version of the categorical characterization of discrete locales in \cite{joyal-tierney} in the context of formal topology:

\begin{definition}[discrete formal topology]
A formal topology is \emph{discrete} if every element is a join of atoms.
\end{definition}
And, of course,  we can prove:

\begin{corollary}\label{eqatom} Let $\P$ be a formal topology. The following are equivalent:
\begin{enumerate}
\item there exists a binary relation $\sqoverlap$ on $\P$ such that $(\P,\leq,\sqoverlap)$ is an atomic overlap algebra;
\item there exists a set $S$ such that $(\P,\leq)$ is order-isomorphic to $(\P(S),\sub)$;
\item $\P$ is a discrete formal topology.
\end{enumerate}
\end{corollary}
\begin{proof}
Just note that to pass from (3) to (1)
one defines an
overlap relation by: $x\sqoverlap y$ iff
 $\Pos(x\wedge
y)$.
\end{proof}

Now, we are going to show how the abstract characterization of discrete
locales in section 5 page 40 of \cite{joyal-tierney} is equivalent
to our notion of discrete formal topology.
Given that the  mentioned characterization of discrete locales  makes reference
to the diagonal map  $\Delta_\P=\langle id_\P,id_\P\rangle$ : $\P$
$\rightarrow$ $\P\times \P$ in the category of locales and given that we do not know how to build predicative
binary products in the whole category of formal topologies but we only know it
 in the full sub-category of inductively generated
 formal topologies in the sense of \cite{cssv}  (see \cite{mv04}), here  we restrict our attention
to inductively generated formal topologies. We just recall that if $S$ is a base for $\P$, then $S\times S$ is a base for the product formal topology $\P\times \P$ and  the corresponding  positivity predicate is $\Pos(p)\, \equiv\, \exists_{a\in S, b\in S}\ (\, (a,b)\leq p \ \& \ (\, \Pos(a)\ \& \ \Pos(b)\, )\,  )$ for $p$ in $\P\times \P$.
In this context we define the notion of
{\em open} formal topology map as follows.
First of all recall that a formal topology map
 $f:\P\rightarrow \Q$ is just a  frame map $f^*:\Q\rightarrow \P$,
 namely a function preserving finite meets and arbitrary joins. A formal topology map
 $f:\P\rightarrow \Q$ is {\em open} if
 the corresponding frame map $f^*:\Q\rightarrow \P$, seen as a functor from
the poset category $\Q$ to $\P$ (i.e. a monotone map),
has a left adjoint  $\exists_f$
satisfying $\exists_f(x\wedge f^*(y))$ = $\exists_f(x)\wedge y$
 (Frobenius reciprocity condition) for all $x$ and
$y$.

 In the next, we will make use of openess of binary product projections in the following form: given $\P$ inductively generated formal topology with base $S$

for $a,b\releps S$ with $a,b$ positive and $W$ subset of $S\times S$

\begin{equation}\label{openess}
(a,b)\leq  \bigvee W  \Rightarrow \left\{
\begin{array}{l} 
 a\, \leq \, \bigvee\{
x\releps S \mid  \exists y\releps S\, (x,y) \releps W\} \\
b\, \leq \, \bigvee\{y\releps S \mid  \exists x\releps S \, (x,y) \releps W\}
\end{array}\right.\end{equation}
(which can be proved by induction, see \cite{mv04}).

Moreover, observe that if $a\releps At(\P)$, then also $(a,a)\releps At(\P\times\P)$: if $(x,y)$ is positive and $(x,y)\leq(a,a)$, then by condition (\ref{openess}) $x\leq a$ and $y\leq a$, and hence $x=y=a$, being $a$ an atom. 

\begin{proposition}
Assuming that $\P$ is an inductively generated formal topology with base $S$,
the following are equivalent:
\begin{enumerate}
\item $\P$ is discrete;
\item the diagonal map $\Delta_\P$ : $\P$
$\rightarrow$ $\P\times \P$
 in the category of formal topologies is open.
\end{enumerate}
\end{proposition}
\begin{proof}
(1$\Rightarrow$ 2)
Thanks to corollary~\ref{eqatom} we think of $\P$ as an overlap algebra and hence we use the characterization of atom in lemma~\ref{lemma char. atoms}.

We define the left adjoint as follows: for $p$ in $\P$
$$\exists_{\Delta_\P}(p)\, \equiv\, \bigvee\{\, (a,a)\releps S\times S \ \mid\  a\leq p\ \& \ a\releps At(\P)\, \}$$
 Indeed, the counit disequality  $\exists_{\Delta_\P}(\, \Delta_\P^*(q)\, )\leq q$ for $q$ in $\P\times \P$ follows easily. First recall from \cite{mv04} that $\Delta_\P^*$ is defined on the base by
 $$\Delta_\P^*(\, (a,b)\, )\,
 \equiv\, a\wedge b\, \equiv\,
\bigvee\{\, x\releps S\, \mid\, x\leq a\wedge b\  \}$$
 and hence extended to the whole $\P$ by
 $\Delta_\P^*(p)=\bigvee\{\Delta_\P^*(a,b)\, \mid\, (a,b)\leq p\}$.\footnote{Note that it is  a join of a set indexed family.} Now,
 if $a$ is an atom satisfying $a\leq \Delta_\P^*(q)$, then there exists $(x,y)\releps S\times S$ such that $a\leq x\wedge y$ with
$(x,y)\leq q$, from which it follows that $(a,a)\leq (x,y)$ and hence
$(a,a)\leq q$.

The unit disequality $p \leq \Delta_\P^*(\, \exists_{\Delta_\P}(p)\, )$ for $p$ in $\P$ follows by density: given an atom $a$ such that $a \sqoverlap p$ then
$a\leq p$ and hence $(a,a)\leq\exists_{\Delta_\P}(p)$ and finally, since $a\leq a \wedge a$, we conclude  $a \leq
 \Delta_\P^*(\, \exists_{\Delta_\P}(p)\, )$, or equivalently $a \sqoverlap
 \Delta_\P^*(\, \exists_{\Delta_\P}(p)\, )$.

In order to prove Frobenius reciprocity, it is enough to show  the disequality $ \exists_{\Delta_\P}(p)\wedge r\ \leq\ \exists_{\Delta_\P}
(\, p\wedge \Delta_\P^*(r)\, )$ for $q$ in $\P$ 
 and $r$ in
$\P\times \P$. 
First of all note that by distributivity $\exists_{\Delta_\P}(p)\wedge r=
\bigvee\{\, (a,a)\wedge r\mid \, a\releps At(\P)\, \&\, a\leq p\, \}$. Hence, to prove the above disequality it is enough to show  that for any positive base element $(x,y)\releps S\times S$ such that $(x,y)\leq (a,a)\wedge r$ with $a$ atom
and $a\leq p$ we have
$(x,y) \leq  \exists_{\Delta_\P}(\, p\wedge \Delta_\P^*(r)\, )$.
 Now, recalling that if $a$ is an atom, then $(a,a)$ is an atom, too, 
from $(x,y)\leq (a,a)$, being $(x,y)$ positive, we get  $(x,y)=(a,a)$
and by condition (\ref{openess}) also that $x=y=a$, and hence $(a,a)\leq r$.
So  
 $a= \Delta_\P^*(\,(a,a)\, )\leq  \Delta_\P^*(r)$, and hence $a\leq p\wedge \Delta_\P^*(r)$. This lets us conclude  $(x,y)=(a,a)\leq \exists_{\Delta_\P}
(\, p\wedge \Delta_\P^*(r)\, )$ by the definition of $\exists_{\Delta_\P}$.

(2$\Rightarrow$1) Let us call $1$ the top of $\P$.
We claim that  any positive base element $(a,b)\releps S\times S$ satisfying $(a,b)\leq \exists_{\Delta_\P}(1)$  is a square,  namely $a=b$.
From the hypothesis $(a,b)\leq \exists_{\Delta_\P}(1) \wedge (a,b)$,  by Frobenius reciprocity we get $(a,b)\leq\exists_{\Delta_\P}(\,\Delta_\P^*(\,(a,b)\, )\, )$ and hence $(a,b)\leq\exists_{\Delta_\P}(a\wedge b)$.
Now, $a\wedge b$ = $\bigvee\{\, c\releps S\ \mid\ c\leq a\wedge b\ \} $. Therefore,
since $ \exists_{\Delta_\P}$ preserves joins (being a left adjoint), we
get
$(a,b)\leq \bigvee\{\, \exists_{\Delta_\P}(c)\ \mid\  c\releps S\ \& \ c\leq a\wedge b\ \}$. Finally, since $\exists_{\Delta_\P}(c)$ = $\exists_{\Delta_\P}(c\wedge c)$ = $\exists_{\Delta_\P}(\,\Delta_\P^*(\,(c,c)\, )\, )$ $\leq$ $(c,c)$ by the counit disequality, we conclude $(a,b)\leq \bigvee\{(c,c)\ \mid\  c\releps S\ \& \ c\leq a\wedge b\ \}$. So, from $\Pos(\, (a,b)\, )$ and from
$(a,b)\leq \bigvee\{(c,c)\ \mid\  c\releps S\ \& \ c\leq a\wedge b\ \}$, we have $a\leq a\wedge b$ and $b\leq a\wedge b$ by condition (\ref{openess}), that is $a=b$, as claimed.

Now, it follows that any positive $a\releps S$ with
$(a,a)\leq \exists_{\Delta_\P}(1)$ is an atom of $\P$: indeed, for $x\releps S$, $x$ positive and
$x\leq a$, we get $(x,a)\leq (a,a)$ and hence $(x,a)\leq \exists_{\Delta_\P}(1)$; so, by what shown above, $x=a$. 
Finally, the unit  disequality $p \leq \Delta_\P^*(\,
\exists_{\Delta_\P}(p)\, )$ for $p$ in $\P$, says that
$p\leq\bigvee\{a\wedge b\ |\ (a,b)\leq\exists_{\Delta_\P}(p)\}$. By
the positivity axiom on $\P$, in the above join 
we can assume all $a\wedge b$ to be positive from which we get that both $a$ and $b$ are so, and hence
 $\Pos(\, (a,b)\, )$. Thus $a=b$ for each $(a,b)$ in the above join. In other words, we obtain $p\leq\bigvee\{a\ |\ a\releps At(\P)\ \& \ (a,a)\leq\exists_{\Delta_\P}(p)\}$. We claim that, for an
atom $a$, the condition $(a,a)\leq\exists_{\Delta_\P}(p)$
implies $a\leq p$. We argue as follows. Since $S$ is a base for
$\P$, one has $p=\bigvee\{x\releps S\ |\ x\leq p\}$ and hence, since $\exists_{\Delta_\P}$ preserves joins,
$(a,a)\leq\bigvee\{\exists_{\Delta_\P}(x)\ |\ x\releps S\ \land\
x\leq p\}$. But $(a,a)\releps At(\P\times\P)$, since $a$ is an atom, so
$(a,a)\leq\exists_{\Delta_\P}(x)$ for some $x\leq p$. Therefore, $a$
= $\Delta_\P^*(a,a)$ $\leq$ $\Delta_\P^*\exists_{\Delta_\P}(x)$ =
$\Delta_\P^*\exists_{\Delta_\P}\Delta_\P^*(x,x)$ =
$\Delta_\P^*(x,x)$ = $x$ $\leq p$, as wished. Summing up, we have
got that $p$ is the join of the atoms below it.
 \end{proof}

\subsection{Free o-algebras}

It is well known that $\P(X)$ is the free suplattice (complete join-semilattice) over a set $X$. The following proposition shows that
$\P(X)$ is also the free o-algebra on a set $X$ of
join-generators.\footnote{Though $\mathbf{OA}$ and $\mathbf{cBa}$
share the same objects, they are very different as categories. For
instance, free complete Boolean algebras generally do not exist (see
\cite{Joh}).}

\begin{proposition}
\label{prop. free o-alg} For any o-algebra $\Q$, any set $X$ and any
map $f:X\rightarrow\Q$, there exists a unique o-morphism
$\overline{f}:\P(X)\rightarrow\Q$ such that the following diagram
commutes:
$$
\xymatrix{ X\ar@{^{(}->}[rr]^i \ar[dr]_f &  &  \mathcal{P}(X)\ar[dl]^{\overline{f}}\\
                    & \Q & \\
}
$$ where $i(x)=\{x\}$, for any $x\releps X$.
\end{proposition}
\begin{proof}
For $U\sub X$, let us put $\overline{f}(U)=\bigvee\{f(x)$ $|$
$x\releps U\}$. This definition is compulsory:
$\overline{f}(U)=\overline{f}(\bigcup\{\{x\}$ $|$ $x\releps
U\})=$ (because
$\overline{f}$ must be an o-morphism and hence it has to preserves joins) $\bigvee\{\overline{f}(\{x\})$ $|$ $x\releps U\}=$ (because $\overline{f}\circ i$
must be $f$) $\bigvee\{f(x)$ $|$ $x\releps U\}$. We claim that $\overline{f}$ is symmetrizable. Let
$g:\Q\rightarrow\P(X)$ be the map defined as in equation (\ref{eq.
def. f^-}) (with respect to $\overline{f}$). Since $\P(X)$ is based
on singletons and $\{x\}\overlap\{a\}$ simply means $x=a$, we can
simplify the expression defining $g$ and get $g(q)=\{x\releps X$ $|$
$f(x)\sqoverlap q\}$. For all $U\sub X$ and $q:\Q$, the following
hold: $ U \overlap g(q)$ $\Leftrightarrow$ $(\exists x\releps
U)(x\releps g(q))$ $\Leftrightarrow$ $(\exists x\releps
U)(f(x)\sqoverlap q)$ $\Leftrightarrow$ $\bigvee\{f(x)$ $|$
$x\releps U\}\sqoverlap q$ $\Leftrightarrow$
$\overline{f}(U)\sqoverlap q$. Thus, $\overline{f}$ is an
o-morphism. Moreover, for all $x\releps X$, $(\overline{f}\circ
i)(x)=\overline{f}(\{x\})=\bigvee\{f(y)$ $|$ $y\releps\{x\}\}=f(x)$.
\end{proof}

\section{Overlap lattices with opposite}

\begin{definition}\label{def. oo-lattice}
An \emph{overlap lattice with opposite} (\emph{oo-lattice} for
short) is a quadruple $(\P,\leq,\sqoverlap,-)$ where $(\P,\leq)$ is
a bounded lattice (with 0 and 1 as the bottom and top elements,
respectively), $-$ is a pseudo-complement operation (that is,
$p\wedge q=0$ if and only $p\leq-q$) and $\sqoverlap$ is a binary
relation on $\mathcal{L}$ satisfying the following properties:
  \begin{itemize}
    \item $p\sqoverlap q\ \Longrightarrow\ q\sqoverlap p$\hfill (symmetry)
    \item $p\sqoverlap q\ \Longrightarrow\ p\sqoverlap (p\wedge q)$\hfill (meet closure)
    \item if $\bigvee_{i\releps I}q_i$ exists, then: $p\sqoverlap\bigvee_{i\releps I}q_i\ \Longleftrightarrow\
    (\exists i\releps I)(p\sqoverlap q_i)$\\\phantom{}\hfill (splitting of existing joins)
    \item $(\forall r:\P)(r\sqoverlap p\ \Rightarrow\ r\sqoverlap q)\ \Longrightarrow\ p\leq q$\hfill(density).
  \end{itemize}
\end{definition}

A set $S$ of elements of an oo-lattice $\P$ is a \emph{base} for $\P$ if
for every $p$ in $\P$ the join of the family $\{a\releps S$ $|$ $a\leq
p\}$ exists and equals $p$. From now on, as usual, we shall always work
with set-based structures. It is easy to check that all properties
stated in Proposition \ref{properties o-algebra} still hold for
oo-lattices. Each (set-based) o-algebra is an
example of oo-lattice: it is enough to define the opposite of an
element $p$ as $\bigvee\{a\releps S$ $|$ $a\wedge p\leq 0\}$. Vice
versa, any oo-lattice that is complete (as a lattice) is
automatically an o-algebra. In the final section of the paper we
shall present several examples of oo-lattices which are not
o-algebras.

Like o-algebras are always locales, so oo-lattices
are always distributive.

\begin{proposition}
Every oo-lattice is a distributive lattice.
\end{proposition}
\begin{proof}
This proof is essentially the finitary version of the first part of
that of Proposition \ref{prop o-lagebras vs overt locales}.
\end{proof}

\begin{remark} By adapting the proof of Proposition
\ref{prop o-lagebras vs overt locales}, it is easy to obtain the following
strengthening of the previous proposition: if $\bigvee_{i\releps I} q_i$ exists in an oo-lattice, then also
$\bigvee_{i\releps I}(p\wedge q_i)$ exists and is equal to
$p\wedge\bigvee_{i\releps I} q_i$.
\end{remark}

\begin{proposition}\label{clasposet}
Classically, the notion of oo-lattice and that of Boolean algebra
coincide.
\end{proposition}
\begin{proof}
The proof is analogous to that of Proposition \ref{prop. o-algebras
cBa's}. Given a Boolean algebra, define $p\sqoverlap q$ as $p\wedge
q\neq 0$ and use Lemma \ref{lemma 2} to prove density. Conversely,
suppose to have an oo-lattice. By the previous Proposition, an
oo-lattice is a distributive lattice. Moreover, by Lemmas \ref{lemma
1} and \ref{lemma 2} we get that the pseudo-complement $-$ is an involution. Summing up,
from a classical point of view an oo-lattice is precisely a complemented
distributive lattice, that is, a Boolean algebra.
\end{proof}

\subsection{Morphisms between oo-lattices}

\begin{definition}
A morphism of oo-lattices from
$(\P,\leq,\sqoverlap,-)$, with base $S$, to $(\Q,\leq,\sqoverlap,-)$ is a map
$f:\P\rightarrow\Q$ such that
$$
f(p)\sqoverlap q\ \Longleftrightarrow\ (\exists a\releps
S)\big(p\sqoverlap a \ \land \ (\forall x\releps S)(x\sqoverlap a\
\Rightarrow\ f(x)\sqoverlap q)\big)
$$
(condition 3 of Proposition \ref{prop. char. o-morphism}), for all
$p:\P$ and $q:\Q$.
\end{definition}

\begin{lemma}
A map $f$ is a morphism of oo-lattices if and only if
\begin{enumerate}
    \item $f$ is monotone and
    \item $f(p)\sqoverlap q\ \Longrightarrow\ (\exists a\releps S)\big(p\sqoverlap
a \ \land \ (\forall x\releps S)(x\sqoverlap a\ \Rightarrow\
f(x)\sqoverlap q)\big)$.
\end{enumerate}
\end{lemma}
\begin{proof}
We firstly prove that each morphism $f$ is monotone. Let $p,r:\P$ be
such that $p\leq r$. We prove that $f(p)\leq f(r)$ by density. Supposed $T$ base of $\Q$,  let $y\releps T$ be such that $f(p)\sqoverlap y$. Then, there
exists $a\releps S$ such that $p\sqoverlap a$ and $(\forall x\releps
S)(x\sqoverlap a$ $\Rightarrow$ $f(x)\sqoverlap y)$. Since $p\leq
r$, then $r\sqoverlap a$. Summing up, there
exists $a\releps S$ such that $r\sqoverlap a$ and $(\forall x\releps
S)(x\sqoverlap a$ $\Rightarrow$ $f(x)\sqoverlap y)$, that is, $f(r)\sqoverlap y$.

Assume now $1$ and $2$. Let $p:\P$ and $q:\Q$ be such that $(\exists
a\releps S)(p\sqoverlap a$ $\land$ $(\forall x\releps S)(x\sqoverlap
a$ $\Rightarrow$ $f(x)\sqoverlap q))$. We claim that $f(p)\sqoverlap q$. Since $S$ is a base for $\P$,
from $p\sqoverlap a$ it follows that there exists $x\releps S$ such
that $x\sqoverlap a$, hence $f(x)\sqoverlap q$, and $x\leq p$. By
monotonicity of $f$, we get $f(x)\leq f(p)$ which, together with
$f(x)\sqoverlap q$, gives the claim.
\end{proof}

\begin{proposition}
\label{consegRED} Let $f:\P\rightarrow\Q$ be a morphism between two
oo-lattices. If $\bigvee_{i\releps I}p_i$ exists in $\P$, then also
$\bigvee_{i\releps I}f(p_i)$ exists and $f(\bigvee_{i\releps I}p_i)$
$=$ $\bigvee_{i\releps I}f(p_i)$.
\end{proposition}
\begin{proof}
We claim that $f(\bigvee_{i\releps I}p_i)$ is the least upper bound
of the family $\{f(p_i)\}_{i\releps I}$. Clearly it is an upper
bound since $f$ is monotone. Let $r$ be another upper bound (that
is, $f(p_i)\leq r$ for any $i\releps I$); we must show that
$f(\bigvee_{i\releps I}p_i)\leq r$. Supposed $T$ base of $\Q$, 
let $y\releps T$ be such that
$f(\bigvee_{i\releps I}p_i)\sqoverlap y$. Then, since $f$ is a
morphism and $\sqoverlap$ splits all existing joins, there exist
$a\releps S$ and $i\releps I$ such that $p_i\sqoverlap a$ and
$(\forall x\releps S)(x\sqoverlap a$ $\Rightarrow$ $f(x)\sqoverlap
y)$. In particular, $f(p_i)\sqoverlap y$ (take $x\leq p_i$ such that
$x\sqoverlap a$ and use monotonicity of $f$); together with
$f(p_i)\leq r$, this gives $r\sqoverlap y$. The claim follows by
density.
\end{proof}

\begin{proposition}
The following hold:
\begin{enumerate}
    \item for every oo-lattice $\P$, the identity function $id_\P:\P\rightarrow\P$ is an oo-lattice morphism;
    \item oo-lattice morphisms are closed under composition of functions.
\end{enumerate}
\end{proposition}
\begin{proof}
(1) Let $p,q:\P$ be such that $p\sqoverlap q$. We must show that
$(\exists a\releps S)$ $(p\sqoverlap a$ $\land$ $(\forall x\releps
S)(x\sqoverlap a$ $\Rightarrow$ $x\sqoverlap q))$, that is,
$(\exists a\releps S)(p\sqoverlap a$ $\land$ $a\leq q)$; that
holds because $S$ is a base.

(2) Let $f:\P\rightarrow\Q$ and $g:\Q\rightarrow\mathcal{R}$ be two
oo-lattice morphisms and assume $g(f(p))\sqoverlap r$. Provided that
$T$ is a base for $\Q$, since $g$ is a morphism we can find an
element $b\releps T$ such that $f(p)\sqoverlap b$ and $(\forall
y\releps T)(y\sqoverlap b$ $\Rightarrow$ $g(y)\sqoverlap r)$. Since
$f(p)\sqoverlap b$ and $f$ is a morphism, there exists $a\releps S$
(where $S$ is a base for $\P$) such that $p\sqoverlap a$ and
$(\forall x\releps S)(x\sqoverlap a$ $\Rightarrow$ $f(x)\sqoverlap
b)$. We claim that this same element $a\releps S$ works for $g\circ
f$, that is, $x\sqoverlap a$ entails $(g\circ f)(x)\sqoverlap r$. In fact, if $x\sqoverlap a$, then
$f(x)\sqoverlap b$. Then we can find $y\releps T$ such that $y\leq f(x)$ and
$y\sqoverlap b$ (since $T$ is a base) and hence $g(y)\sqoverlap r$
(thanks to the properties of $b$). On the other hand $g(y)\leq
g(f(x))$ because $g$ is monotone; hence
$g(f(x))\sqoverlap r$.
\end{proof}

\begin{definition}
Let \textbf{OOLat} be the category of oo-lattices as objects and
oo-lattice morphisms as arrows.
\end{definition}

\begin{proposition}\label{prop. CLASS char. oo-morphism}
Classically and impredicatively, a map between Boolean algebras is a
morphism of oo-lattices if and only if it preserves all joins which
exist in the domain.
\end{proposition}
\begin{proof}
From Proposition \ref{consegRED} it follows that every oo-lattice
morphism is a join-preserving map. Conversely, let
$f:\mathcal{B}\rightarrow\mathcal{B'}$ be a join-preserving map
between two Boolean algebras. Then $f$ extends uniquely to a
join-preserving map $\overline{f}:DM(\mathcal{B})\rightarrow
DM(\mathcal{B'})$ between the complete boolean algebras
defined as the Dedekind-MacNeille completions  of $\mathcal{B}$ and $\mathcal{B'}$, respectively (see
\cite{Joh}). It
follows from Propositions \ref{prop. CLASS: o-morphism =
join-preserving} and \ref{prop. char. o-morphism} that
$\overline{f}$ satisfies equation (\ref{RED}) for all
$p:DM(\mathcal{B})$ and $q:DM(\mathcal{B'})$ and for any base $S$ of
$DM(\mathcal{B})$. In particular, (\ref{RED}) holds for $f$ once we
note that every base $S$ for $\B$ is also a base for $DM(\B)$.
\end{proof}

\begin{definition}
We write \textbf{Ba}$_{\bigvee}$ for the category of Boolean
algebras and maps preserving all existing joins.
\end{definition}

\begin{corollary}
Classically, the categories \textbf{Ba}$_{\bigvee}$ and
\textbf{OOLat} are equivalent.
\end{corollary}
\begin{proof}
By Propositions \ref{clasposet} and \ref{prop. CLASS char.
oo-morphism}.
\end{proof}

\subsection{Richer overlap structures}

As we have seen in the previous pages, the notion of oo-lattice
turns out to be classically equivalent to that of Boolean algebra;
so it is a constructive version of the latter. Even if it seems the
minimal structure enjoying such a property, it is by no means the
only one. For instance, it is quite natural to consider also Heyting
and Boolean algebras with overlap (o-Heyting and o-Boolean
algebras). The idea is simply to add an overlap relation (satisfying
all the axioms for $\sqoverlap$ listed in Definition \ref{def.
oo-lattice}) to the usual algebraic structures.

\begin{definition}\label{def. o-structures}
An \emph{overlap Heyting algebra} (\emph{o-Ha} for short) is an
oo-lattice whose underlying lattice is a Heyting algebra. An
\emph{overlap Boolean algebra} (\emph{o-Ba} for short) is an o-Ha
whose underlying lattice is a Boolean algebra.
\end{definition}

Here is an example of o-Ba (examples of other o-structures are being
given below): the family of all \emph{recursive} subsets of
$\mathbb{N}$ (the set of natural numbers). It works since recursive
subsets are closed under union, intersection and complement;
moreover, all singletons are recursive, hence density holds. Note
also, that this o-Ba is an example of oo-lattice which is not an
o-algebra (that is, it is not closed under arbitrary joins),
otherwise all subsets of $\mathbb{N}$ would be recursive (each
subset being a union of singletons).

\begin{proposition}\label{prop class char o-structures}
Classically, the notions of oo-lattice, o-Ha, o-Ba and that of
Boolean algebra all coincide.
\end{proposition}
\begin{proof}
By the proof of Proposition \ref{clasposet}.
\end{proof}

We adopt for all o-structures the same notion of morphism we used
for oo-lattices.

\begin{definition}
Let $\mathbf{OHa}$ be the full subcategory of $\mathbf{OOLat}$ whose
objects are the o-Heyting algebras. We write $\mathbf{OBa}$ for the
full subcategory of $\mathbf{OHa}$ with o-Boolean algebras as
objects.
\end{definition}

\begin{proposition}
Classically, the categories \textbf{Ba}$_{\bigvee}$, \textbf{OOLat},
\textbf{OHa} and \textbf{OBa} are all equivalent.
\end{proposition}
\begin{proof}
By Propositions \ref{prop class char o-structures} and \ref{prop.
CLASS char. oo-morphism}.
\end{proof}

On the contrary, from a constructive point of view, the situation is
completely different and can be summarized by the following picture.
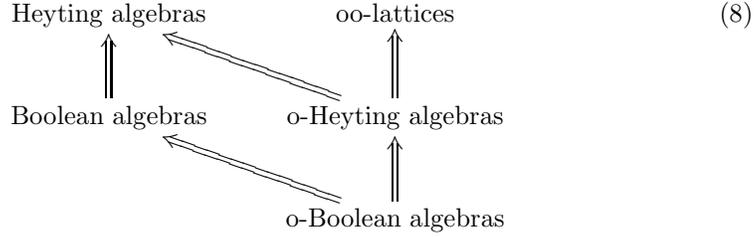
\begin{equation}\label{figure o-structures}
\xymatrix{ \textrm{Heyting algebras} & \textrm{oo-lattices}\\
                \textrm{Boolean algebras}\ar@{=>}[u]  &
                \textrm{o-Heyting algebras}\ar@{=>}[u]\ar@{=>}[ul] & \\
                 & \textrm{o-Boolean algebras}\ar@{=>}[u]\ar@{=>}[ul]
}
\end{equation}
No other ``implication'' holds constructively, as it is shown by the
following counterexamples and remarks.

Here below is an example of o-Ha which is not a Boolean algebra, constructively. Classically, this
turns out to be nothing else but the Boolean algebra of finite-cofinite
subsets. Let $X$ be a set. We say that a subset $K\sub X$ is
\emph{finite} if either $K=\emptyset$ or $K=\{x_1,\ldots,x_n\}$ for
some $x_1\ldots x_n\releps X$. Clearly, the union of two finite
subsets is finite, while the intersection is not (unless the
equality of $X$ is decidable: see \cite{cisa2}).

\begin{definition}\label{def. finite-cofinite oo-lattice}
For any set $X$, let $\F(X)$ be the sub-family of $\P(X)$ defined by
the following condition (for $A\sub X$):
\begin{equation}\label{eq. def. finite-cofinite
oo-lattice}A:\F(X)\quad\stackrel{def}{\Longleftrightarrow}\quad(\
\exists K\sub X,\ K\textrm{ finite}\ )(\ A\sub--K\ \vee\ -K\sub A\
)\,.\end{equation}
\end{definition}

\begin{proposition}
For every set $X$, the collection $\F(X)$ is an o-Heyting algebra, but
it is neither a Boolean algebra nor an o-algebra, in general.
\end{proposition}
\begin{proof}
$\F(X)$ contains both $\emptyset$ (which is finite) and $X$ (because
$X$ $=$ $-\emptyset$).

$\F(X)$ is closed under union: let $A,B:\F(X)$; if either $A$ or $B$
contains some cofinite subset, then so does $A\cup B$; otherwise,
there exist two finite subsets $K$ and $L$ such that $A\sub--K$ and
$B\sub--L$; so $A\cup B$ $\sub$ $--K\cup--L$ $\sub$ $--(K\cup L)$.

$\F(X)$ is closed under intersection: let $A,B:\F(X)$; if either $A$
or $B$ is contained in some \emph{co}cofinite subset, then so is
$A\cap B$; otherwise, there exist two finite subsets $K$ and $L$
such that $-K\sub A$ and $-L\sub B$; so $-(K\cup L)=-K\cap-L\sub
A\cap B$.

This proves that $\F(X)$ is a sublattice of $\P(X)$. Since $\P(X)$
is an o-algebra, its overlap relation $\overlap$, when restricted to
a sublattice, automatically inherits all properties required by
Definition \ref{def. oo-lattice} but at most density. However, density holds
for $\F(X)$ as well as for any other sublattice of $\P(X)$ which
contains all singletons (as $\F(X)$ clearly does).\footnote{By the way, note that $F(X)$ contains all finite subsets of $X$, as well as all cofinite and also all cocofinite ones.}

$\F(X)$ is closed under implication\footnote{Remember that $\P(X)$
is a Heyting algebra with $A\rightarrow B$ $=$ $\{x\releps X$ $|$
$A\cap\{x\}\sub B\}$.}: let $A,B:\F(X)$; if $-L\sub B$, then $-L\sub A\rightarrow B$
(because $B\sub A\rightarrow B$); if $A\sub --K$, then $-K\sub -A$,
hence $-K\sub A\rightarrow B$ (because $-A\sub A\rightarrow B$);
finally, if $B\sub --L$ and $-K\sub A$, then $-(K\cup L)$ $=$
$-K\cap-L$ $\sub$ $A\cap-B$ $\sub$ $-(A\rightarrow B)$; hence
$A\rightarrow B\sub--(K\cup L)$. So $A\rightarrow B:\F(X)$
in any case.

Thus $\F(X)$ is an o-Ha. Clearly, it is not an o-Ba, in general; for
instance, $\{x\}\cup-\{x\}$ need not be equal to $X$ (unless $X$ has
a decidable equality).

Finally, $\F(X)$ is clearly not complete (it is not an o-algebra):
think of $X$ $=$ $\mathbb{N}$ and consider the elements
$\{2n\}:\F(\mathbb{N})$, for $n\releps\mathbb{N}$; their union is
the set of all even numbers which, of course, does not belong to
$\F(\mathbb{N})$.
\end{proof}

The notion of Boolean algebra seems constructively weaker than its
overlap version: at least, the relation $p\wedge q\neq 0$ (which
seems to be the only possible candidate) fails to be an overlap
relation constructively. In fact, if that were the case, then every Boolean
algebra would have a stable equality \footnote{That means $\neg\neg(x=y)\Rightarrow (x=y)$ for all $x,y$.} (by Lemma \ref{lemma 2}).
A fortiori, one cannot hope to find a general method for endowing a
Heyting algebra with an overlap relation.

Finally, we give an example of oo-lattice which does not seem to be
an o-Heyting algebra. Consider a pure first-order language with
equality and define the smallest class of formulae which contains
atomic formulae and is closed under disjunction, conjunction and
negation. Given a set $X$, the family of all its subsets that can be
obtained by comprehension on those formulae is an
oo-lattice,\footnote{In other words, this is the smallest sub-family
of $\P(X)$ that contains all singletons and is closed under finite
unions, finite intersections and pseudo-complementation.
Classically, this is nothing else but another description of the
Boolean algebra of finite-cofinite subsets.} but there seems to be
no constructive way to define implication.

\section*{Conclusions and future work}

In this paper we have done the following:
\begin{list}{-}{ }


\item we have shown that, classically,  the category of overlap algebras is
equivalent to the category of complete Boolean algebras and
join-preserving maps;

\item we have defined the notion of discrete formal topology and compare it
with the notion of atomic set-based overlap algebra and the categorical description of  discrete locales
in \cite{joyal-tierney};

\item we have shown constructively that the power-collection of a set is the free overlap algebra generated from that set;

\item we have introduced various structures equipped with an overlap relation, called {\it o-structures}, and corresponding morphisms, which generalize Sambin's notion of
overlap algebra and of overlap morphism, respectively;

\item we have shown that the corresponding categories of our o-structures
are all equivalent to the category of Boolean algebras with maps
preserving existing joins.
\end{list}

In the future, we aim to test whether our o-structures can be used
to give a constructive version of Stone representation for (not
necessarily complete) Boolean algebras and to investigate the
existence of constructive join-completions for them.

\section*{Acknowledgements}

We express our gratitude to Giovanni Sambin and Steve Vickers for
very interesting discussions about overlap algebras and their
interplay with Boolean algebras.

\bibliographystyle{plain}
\bibliography{cmt}

\end{document}